\documentclass[11pt]{article}

\usepackage[latin1]{inputenc}
\usepackage{t1enc}
\usepackage{amsmath, amssymb, amsfonts, amsthm, amsopn,epsfig}
\usepackage[none]{hyphenat}
\usepackage{tikz}
\usepackage{float}
\usepackage{graphicx}
\usepackage{caption}
\usepackage[toc,page]{appendix}
\usepackage{epstopdf}
\usepackage{etoolbox}
\usepackage[colorlinks=true]{hyperref}
\usepackage{dsfont}
\hypersetup{
	colorlinks=true,
	linkcolor=blue,
	filecolor=magenta,      
	urlcolor=cyan,
	citecolor=blue
}
\usepackage{cleveref}
\usepackage[top=24mm, bottom=23mm, left=25mm, right=25mm]{geometry}

\theoremstyle{plain}
\newtheorem{theorem}{Theorem}[section]

\newtheorem{lemma}[theorem]{Lemma}

\theoremstyle{definition}

\DeclareMathOperator{\perm}{perm}

\title{The number of symmetric chain decompositions}
\author{Istv\'an Tomon\thanks{Ume\r{a} University, \emph{e-mail}: \textbf{istvantomon@gmail.com}, Research supported in part by the Swedish Research Council grant VR 2023-03375.}}
\date{}

\begin{document}
	\sloppy 
	
	\maketitle

 \begin{abstract}
     We prove that the number of symmetric chain decompositions of the Boolean lattice $2^{[n]}$ is $$\left(\frac{n}{2e}+o(n)\right)^{2^n}.$$
     Furthermore, the number of symmetric chain decompositions of the hypergrid $[t]^n$ is 
     $$n^{(1-o_n(1))\cdot t^n}.$$
 \end{abstract}

\section{Introduction}
The \emph{Boolean lattice} $2^{[n]}$ is the power set of $[n]=\{1,\dots,n\}$ ordered by inclusion. Decompositions of the Boolean lattice into chains are the subject of extensive study, and so called symmetric chain decompositions (SCD) are of particular interest. A \emph{chain} in $2^{[n]}$ is a sequence of sets $C_0\subset \dots \subset C_k$, and a chain is \emph{symmetric} if $|C_0|+|C_k|=n$ and $|C_i|=|C_0|+i$ for $i\in [k]$. In other words, a chain is symmetric if the sizes of the sets in the chain form an interval with midpoint $n/2$. A decomposition of $2^{[n]}$ into symmetric chains was first constructed by de Brujin, Tengbergen, and Kruyswijk \cite{BTK} based on a recursive argument. Greene and Kleitman \cite{GK} used a parentheses matching approach to construct such decompositions explicitly, while further proofs of existence \cite{Griggs} are also available. Shearer and Kleitman \cite{SK} show that there exist two edge disjoint SCDs in $2^{[n]}$, while Gregor, J\"ager, M\"utze, Sawada, and Wille \cite{GJMSW} find explicit constructions of four edge-disjoint SCDs if $n\geq 12$.

The importance of symmetric chain decompositions comes from their profound applications. It is easy to show that every decomposition of $2^{[n]}$ into symmetric chains contains exactly $\binom{n}{\lfloor n/2\rfloor}$ chains, which is the minimum number of chains in any chain decomposition, giving an alternative proof of Sperner's theorem \cite{Sperner}. Strong versions of Sperner's theorem easily follow from the existence of SCDs as well \cite{Katona}. Furthermore, constructions of SCDs played an important role in the resolution of the vector Littlewood-Offord problem by Kleitman \cite{kleitman65}, and the existence of so called symmetric Venn diagrams also relies on SCDs \cite{GKS}. For more recent applications, see \cite{BTW,GJMSW}.

However, not much is known about the total number of different SCDs in $2^{[n]}$. Are there many different SCDs or do all SCDs arise from some of the explicit constructions mentioned above? By observing that the bipartite comparability graph between the consecutive levels of $2^{[n]}$ has maximum degree at most $n$, it is easy to argue that the number of SCDs is less than $n^{2^n}$. On the other hand, the author of this paper \cite{Tomon} established the lower bound $n^{\Omega(2^n/\sqrt{n})}$. More precisely, the proof of \cite{Tomon} implies that if $E\subset 2^{[n]}$ is the set of endpoints of the chains in an SCD, then $E$ can already take $n^{\Omega(2^n/\sqrt{n})}$ different values. As $|E|= 2\binom{n}{\lfloor n/2\rfloor}$ if $n$ is odd, and $|E|=\binom{n}{ n/2}+\binom{n}{ n/2-1}$ if $n$ is even, we also have the matching upper bound $\binom{2^n}{|E|}=n^{O(2^n/\sqrt{n})}$ for this counting problem. The main result of our paper gives a surprisingly sharp estimate for the number of SCDs, showing that the trivial upper bound $n^{2^{n}}$ is not far from the truth. This also shows that SCDs are much more common than one might initially believe.

\begin{theorem}\label{thm:Boolean}
The number of symmetric chain decompositions of $2^{[n]}$ is 
$$\left(\frac{n}{2e}+o(n)\right)^{2^n}.$$
\end{theorem}

A natural generalization of the Boolean lattice is the \emph{hypergrid}, sometimes referred to as the \emph{divisor lattice}. The hypergrid is the set $[t]^n$ endowed with the coordinate-wise ordering $\prec$, that is, for $(x_1,\dots,x_n),(y_1,\dots,y_n)\in [t]^n$, we have $(x_1,\dots,x_n)\preceq (y_1,\dots,y_n)$ if $x_i\leq y_i$ for every $i\in [n]$. This is indeed a generalization, as $[2]^n$ is isomorphic to the Boolean lattice $2^{[n]}$. Many natural properties of the Boolean lattice carry over to the hypergrid as well, however, oftentimes the proof of these properties is much more involved. The reason for this is that while the bipartite graph between consecutive levels of the Boolean lattice is bi-regular, this is no longer the case for the hypergrid if $t\geq 3$. One such property is that $[t]^n$ is Sperner, that is, the size of the largest antichain is equal to  the size of a largest level. Nevertheless, as proved by de Brujin, Tengbergen, and Kruyswijk \cite{BTK}, the hypergrid $[t]^n$ also has a symmetric chain decomposition (see the Preliminaries for formal definitions), which immediately implies that $[t]^n$ satisfies even strong notions of the Sperner property. In \cite{BTK}, an explicit SCD is constructed in a clever recursive manner. This raises the question whether there exist other, significantly different constructions of SCDs, and in particular, how many are there. Similarly as before, it is easy to argue that the number of SCDs of $[t]^n$ is at most $n^{t^n}$, by observing that the bipartite graph between the levels has maximum degree at most $n$. Our second main result shows that this trivial upper bound is not too far from the actual number, so SCDs of $[t]^n$ are fairly abundant.

\begin{theorem}\label{thm:hypergrid}
Let $\varepsilon>0$. If $n>n_0(\varepsilon)$ is sufficiently large, then for every $t\geq 2$, the number of symmetric chain decompositions of $[t]^{n}$ is at least
$$n^{(1-\varepsilon)t^n}.$$
\end{theorem}

We prove Theorems \ref{thm:Boolean} and \ref{thm:hypergrid} in Sections \ref{sect:1} and \ref{sect:2}, respectively, after introducing our notation.
\section{Preliminaries}

In this section, we recall a number of basic definitions related to partial orders. For a detailed overview of the subject, we refer the interested reader to \cite{book_Anderson}.

Let $P$ be a poset with partial ordering $\prec$. An element $y\in P$ \emph{covers} $x\in P$ if $x\prec y$ and there is no $z\in P$ such that $x\prec z \prec y$. The \emph{cover graph} of $P$ is the graph on vertex set $P$ in which two vertices are joined by an edge if one covers the other. The \emph{comparability graph} of $P$ is the graph on vertex set $P$ in which two vertices are joined by an edge if they are comparable by $\prec$. The poset $P$ is \emph{graded} if there exists a partition $L_0,\dots,L_m$ of $P$ such that $L_0$ is the set of minimal elements, and if $y\in L_i$ covers $x\in L_j$, then $i=j+1$. If $P$ is graded, the partition $L_0,\dots,L_m$ is unique, and $L_i$ is called a \emph{level} of $P$. In this case, if $x\in P$, the \emph{rank} of $x$ is the unique index $i$ such that $x\in L_i$, and we denote it by $r(x)$. A graded poset $P$ is \emph{rank-symmetric} if $|L_i|=|L_{m-i}|$ for $i=0,\dots,k$. A chain $x_0\prec \dots\prec x_{k}$ in $P$ is \emph{symmetric} if $r(x_0)+r(x_k)=m$ and $r(x_i)=r(x_0)+i$ for every $i=1,\dots,k$.

Clearly, the Boolean lattice $2^{[n]}$, and more generally the hypergrid $[t]^n$ is graded and rank-symmetric. In particular, $[t]^n$ has $n(t-1)+1$ levels, and if $(x_1,\dots,x_n)\in [t]^n$, then $r((x_1,\dots,x_n))=x_1+\dots+x_n-n$.

\section{Boolean lattice}\label{sect:1}
In this section, we prove Theorem \ref{thm:Boolean}. Our proof relies on some well known inequalities about the number of matchings in bipartite graphs.

A matrix is \emph{doubly stochastic} if every entry is non-negative, and the sum of entries in each row and column is 1. One of the main tools to prove our lower bounds  is a celebrated result of Falikman \cite{Falikman} on the permanents of doubly stochastic matrices, which confirmed a well known conjecture of van der Waerden. See also Egorychev \cite{Egorychev} and Schrijver \cite{Schrijver} for sharper versions.

\begin{theorem}[\cite{Falikman}]\label{lemma:perm}
Let $A$ be a doubly stochastic $n\times n$ matrix. Then
$$\perm(A)\geq \frac{n!}{n^n}>e^{-n}.$$
\end{theorem}

To prove our upper bounds, we apply a result of Br\'egman \cite{Bregman} on the number of matchings in bipartite graphs with given degree sequence.

\begin{theorem}[\cite{Bregman}]\label{lemma:matching_upper}
Let $G$ be a bipartite graph such that the degree sequence of  one of the vertex classes is $d_1,\dots,d_n$. Then the number of perfect matchings in $G$ is at most 
$$\prod_{i=1}^{n}(d_i!)^{1/d_i}.$$
\end{theorem}

Our first technical lemma proves bounds on the number of symmetric chain decompositions of a rank-symmetric poset with three levels. We use a trick of Griggs \cite{Griggs} to identify the symmetric chain decompositions of such a poset with perfect matchings in a carefully constructed bipartite graph. Let us introduce some further notation.

A bipartite graph $G=(X,Y;E)$ is \emph{bi-regular} if every vertex in $X$ has the same degree, and every vertex in $Y$ has the same degree. Say that a graded poset $P$ is \emph{regular} if the bipartite comparability graph between any two consecutive levels is bi-regular. Clearly, $2^{[n]}$ is regular. Given a bipartite graph $G=(X,Y;E)$ with edge-weighting $w:E\rightarrow \mathbb{R}$, the \emph{bi-adjacency matrix} of $(G,w)$ is the $|X|\times |Y|$ matrix $A$, whose rows are indexed by elements of $X$, columns are indexed by elements of $Y$, and $A(x,y)=w(xy)$ if $xy$ is an edge, and $A(x,y)=0$ otherwise.

\begin{lemma}\label{lemma:3level}
Let $P$ be a regular, rank-symmetric poset with three levels $X,Y,Z$, $a=|X|=|Z|< |Y|=b$. Assume that every vertex in $X\cup Z$ is comparable to $r$ elements of $Y$. Then the number of symmetric chain decompositions of $P$ is between 
$$\left(\frac{r}{e}\right)^{2a}\cdot e^{-2(b-a)}\ \ \ \text{ and }\  \ \ \ (r!)^{(a+b)/r}.$$
\end{lemma}

\begin{proof}
 Note that by the regularity of $P$, every vertex in $Y$ has exactly $r\cdot \frac{a}{b}$ neighbours in both $X$ and $Z$.  Let $Y_1$ and $Y_2$ be two disjoint copies of $Y$, and define the edge-weighted bipartite graph $(G,w)$ with vertex classes $U=Y_1\cup X$ and $V=Y_2\cup Z$ as follows. If $y\in Y$ and $x\in X$ are comparable, then add an edge of weight $1/r$ between $x$ and the copy of $y$ in $Y_2$. Similarly, if  $y\in Y$ and $z\in Z$ are comparable, then add an edge of weight $1/r$ between $z$ and the copy of $y$ in $Y_1$. Finally, if $y_1\in Y_1$ and $y_2\in Y_2$ are copies of the same vertex, then add an edge of weight $1-\frac{a}{b}$ between $y_1$ and $y_2$. Writing $A$ for the weighted bi-adjacency matrix of $(G,w)$, we have that $A$ is an $(a+b)\times (a+b)$ sized doubly stochastic matrix.

It is easy to see that there is a bijection between the perfect matchings of $G$ and the symmetric chain decompositions of $P$. Indeed, let $M$ be a perfect matching, then we define the chain decomposition as follows. Let $y\in Y$, and let $y_1\in Y_1$ and $y_2\in Y_2$ be the copies of $y$. If $\{y_1,y_2\}\in M$, then let $y$ form a single vertex chain. Otherwise, there exist $x\in X$ and $z\in Z$ such that $\{x,y_2\}\in M$ and $\{z,y_1\}\in M$, in which case $\{x,y,z\}$ is a chain. This clearly gives a symmetric chain decomposition of $P$. We can argue similarly that every symmetric chain decomposition comes from unique a perfect matching in this manner.

Now let us consider the lower bound. For a matching $M$, let $w(M)=\prod_{e\in M}w(e)$. If $M$ is a perfect matching, then it contains exactly $2a$ edges of weight $1/r$, and $b-a$ edges of weight $1-\frac{a}{b}$. Hence, $w(M)\leq r^{-2a}.$ Furthermore, by Lemma \ref{lemma:perm},
$$\sum_{M}w(M)=\perm(A)\geq e^{-(a+b)},$$
where the sum $\sum_{M}$ is taken over all perfect matchings of $G$. Hence, we conclude that the number of perfect matchings of $G$ is at least
$$r^{2a}e^{-(a+b)}=\left(\frac{r}{e}\right)^{2a}\cdot e^{-2(b-a)}.$$

Now let us turn to the upper bound. Write $q=ar/b+1$, so the vertices of $Y_1$ and $Y_2$ have degree $q$. Note that $q\leq r$ by our assumption that $a<b$. By Lemma \ref{lemma:matching_upper}, the number of perfect matchings of $G$ is at most 
$$(r!)^{a/r}\cdot (q!)^{b/q}\leq (r!)^{(a+b)/r}.$$
Here, we used that the function $(d!)^{1/d}$ is monotone increasing.
\end{proof}

\begin{proof}[Proof of Theorem \ref{thm:Boolean}]
For ease of notation, let us assume that $n$ is even, the case of odd $n$ follows in the same manner. Let $m=n/2$, and let $L_0,\dots,L_n$ be the levels of $2^{[n]}$. Every symmetric chain decomposition of $2^{[n]}$ can be generated as follows. For $s=1,\dots,m$, let $B_s$ be the subposet of $2^{[n]}$ induced by the levels $L_{m-s},\dots,L_{m+s}$. First, we find a symmetric chain decomposition $C_1$ of $B_1$. Then, step-by-step, we extend this to a symmetric chain decomposition  $C_s$ of $B_s$ as follows. The chain decomposition $C_{s-1}$ defines a bijection between the levels $L_{m-s+1}$ and $L_{m+s-1}$ by mapping the end of the chains of $C_{s-1}$ to each other. We identify the elements of $L_{m-s+1}$ and $L_{m+s-1}$ along this bijection. After this, we can view the subposet of $2^{[n]}$ induced on the levels $L_{m-s},L_{m-s+1},L_{m+s-1},L_{m+s}$ as a three level poset $P_s$. Then every symmetric chain decomposition $D_s$ of $P_{s}$ gives a symmetric chain decomposition $C_s$ of $B_{s}$, which extends $C_{s-1}$ in an obvious manner. 

For $s=1,\dots,m$, the poset $P_s$ is bi-regular, where every element not in the middle level has exactly $m+s$ neighbours in the middle level. Hence, it is enough to bound the number of symmetric chain decompositions of $P_s$. Let us start with the lower bound. By Lemma \ref{lemma:3level}, the number of different choices for $D_{s}$ is at least
$$\left(\frac{m+s}{e}\right)^{2|L_{m+s}|}e^{-2(|L_{m+s}|-|L_{m+s-1}|)}\geq \left(\frac{n}{2e}\right)^{2|L_{m+s}|}e^{-2(|L_{m+s}|-|L_{m+s-1}|)}.$$
 Therefore, the number of symmetric chain decompositions is at least
$$\prod_{s=1}^{m}\left(\frac{n}{2e}\right)^{2|L_{m+s}|}e^{-2(|L_{m+s}|-|L_{m+s-1}|)}\geq \left(\frac{n}{2e}\right)^{2^n-|L_m|}\cdot e^{-2|L_m|}=\left(\frac{n}{2e}\right)^{2^n\cdot (1-O(n^{-1/2}))}.$$
Here, we used that $|L_m|=\binom{n}{n/2}=O(2^n/\sqrt{n})$.

Now let us turn to the upper bound. By  Lemma \ref{lemma:3level}, the number of different choices for $D_{s}$ is at most
$$((m+s)!)^{(|L_{m+s-1}|+|L_{m+s}|)/(m+s)}\leq \left((1+o(1))\frac{m+s}{e}\right)^{|L_{m+s-1}|+|L_{m+s}|}.$$
As long as $s\leq n^{2/3}$, the right hand side is at most
$$\left((1+o(1))\frac{n}{2e}\right)^{|L_{m+s-1}|+|L_{m+s}|}.$$
Hence, the total number symmetric chain decompositions is at most 
$$\prod_{s=1}^{n^{2/3}}\left((1+o(1))\frac{n}{2e}\right)^{|L_{m+s-1}|+|L_{m+s}|} \prod_{s=n^{2/3}}^m \left((1+o(1))\frac{n}{e}\right)^{|L_{m+s-1}|+|L_{m+s}|}=\left(\left(\frac{1}{2e}+o(1)\right)n\right)^{2^n}.$$
Here, we used that $\sum_{s=n^{2/3}}^{m}|L_{m+s-1}|+|L_{m+s}|=O(2^n n^{-10})$ by standard concentration arguments.
\end{proof}

\section{The hypergrid}\label{sect:2}

We now turn to the proof of Theorem \ref{thm:hypergrid}. We follow similar ideas as in the previous section. Unfortunately, Lemma \ref{lemma:3level} is no longer applicable as $[t]^n$ is not regular. In order to overcome this, we introduce a weighting of the edges of the cover graph of $[t]^n$, which makes it regular in a certain weighted sense. In order to do this, we need to introduce some notation.

Given a bipartite graph $G=(X,Y;E)$, a function $f:E\rightarrow \mathbb{R}_{\geq 0}$ is a \emph{normalized matching} if there exists $a,b\in \mathbb{R}$ such that for  every $x\in X$, $$\sum_{y\in Y: xy\in E} f(xy)=a$$ and for every $y\in Y$, $$\sum_{x\in X: xy\in E} f(xy)=b.$$
Note that $a$ and $b$ must satisfy $a/b=|Y|/|X|$. The function $f$ is a \emph{scaled normalized matching} if $a=1$ (assuming the order of $X$ and $Y$ is given), which then implies $b=|X|/|Y|$. Given a graded poset $P$ with cover graph $G$, a function $f:E(G)\rightarrow \mathbb{R}_{\geq 0}$ is a \emph{normalized matching flow} if the restriction of $f$ to the bipartite graph between consecutive levels of $P$ is a normalized matching, and $f$ is a \emph{scaled normalized matching flow} (or SNMF, for short), if this restriction is a scaled normalized matching, with the smaller indexed level playing the role of $X$.

Normalized matching flows were introduced by Kleitman \cite{kleitman74}, who proved that their existence is equivalent to other important properties of posets, such as the \emph{LYM property}, and the \emph{normalized matching property}. We refer the interested reader to \cite{book_Anderson} or \cite{kleitman74}. It was proved by Harper \cite{harper74} that $[t]^n$ has a normalized matching flow (and thus all three of the aforementioned properties). However, it was only recently proved by Falgas-Ravry, R\"aty and Tomon \cite{FRT} that one can also find a normalized matching flow in which the weights are somewhat evenly distributed. This additional property is the key to our argument.

\begin{lemma}[Lemma 4.7 in \cite{FRT}]\label{lemma:regular_flow}
There exists an SNMF $f$ on the cover graph of $[t]^n$ such that for every edge $xy$, where $x\prec y$ and $|r(x)-\frac{(t-1)n}{2}|\leq tn^{3/5}$, we have
$$f(xy)=O\left(\frac{\log n}{n}\right).$$
\end{lemma}

In what follows, we prove a variant of the lower bound of Lemma \ref{lemma:3level}.

\begin{lemma}\label{lemma:3level+}
Let $P$ be a symmetric poset with three levels $X,Y,Z$, $a=|X|=|Z|< |Y|=b$, let $f$ be an SNMF on $P$, and let $W$ be the maximum of $f(xy)$ among all edges $xy$ of the cover graph. Then the number of symmetric chain decompositions of $P$ is at least
$$W^{-2a} e^{-(b+a)}.$$
\end{lemma}
\begin{proof}
 Let $Y_1$ and $Y_2$ be two disjoint copies of $Y$, and define the edge-weighted bipartite graph $(G,w)$ with vertex classes $U=Y_1\cup X$ and $V=Y_2\cup Z$ as follows. If $y\in Y$ and $x\in X$ are comparable, then add an edge of weight $f(xy)$ between $x$ and the copy of $y$ in $Y_2$. Similarly, if  $y\in Y$ and $z\in Z$ are comparable, then add an edge of weight $\frac{a}{b}\cdot f(yz)$ between $z$ and the copy of $y$ in $Y_1$. Finally, if $y_1\in Y_1$ and $y_2\in Y_2$ are copies of the same vertex, then add an edge of weight $1-\frac{a}{b}$ between $y_1$ and $y_2$. This weighting has the property that the sum of weights at every vertex is equal to 1, using that $f$ is an SNMF. Hence, writing $A$ for the weighted bi-adjacency matrix of $(G,w)$, we have that $A$ is an $(a+b)\times (a+b)$ sized doubly stochastic matrix. 

 As in the proof of Lemma \ref{lemma:3level}, we observe that the number of perfect matchings of $G$ is equal to the number of symmetric chain decompositions of $P$. For a matching $M$, let $w(M)=\prod_{e\in M}w(e)$. If $M$ is a perfect matching, then it contains $2a$ edges of weight at most $W$, and $b-a$ further edges of weight $1-\frac{a}{b}$. Hence, $w(M)\leq W^{2a}.$
 Furthermore, by Lemma \ref{lemma:perm},
$$\sum_{M}w(M)=\perm(A)\geq e^{-(a+b)},$$
where the sum $\sum_{M}$ is taken over all perfect matchings of $G$. Thus, we conclude that the number of perfect matchings of $G$ is at least
$W^{-2a}e^{-(a+b)}.$
\end{proof}

\begin{proof}[Proof of Theorem \ref{thm:hypergrid}]
For ease of notation, let us assume that $(t-1)n$ is even, the case of odd $(t-1)n$ follows in the same manner. Let $m=(t-1)n/2$, let $L_0,\dots,L_{(t-1)n}$ be the levels of $[t]^{n}$, and let $f$ be an SNMF of $[t]^n$ satisfying the properties of Lemma \ref{lemma:regular_flow}. That is, writing $W$ for the maximum of $f(xy)$ over all edges $xy$ of the cover graph, where $x\prec y$ and $|r(x)-\frac{(t-1)n}{2}|\leq tn^{3/5}$, we have $W=O(\frac{\log n}{n})=n^{-1+o_n(1)}$.

Repeating the proof of Theorem \ref{thm:Boolean}, our goal is to find a lower bound on the number of symmetric chain decompositions of the 3 level poset $P_s$, defined as follows. The middle level $Y$ of $P_s$ is a gluing of the elements of $L_{m-s+1}$ and $L_{m+s-1}$ along some bijection, while the lowest and largest levels are $L_{m-s}$ and $L_{m+s}$, respectively. Clearly, the restriction of $f$ to $P_s$ is also an SNMF. Also, if $s\leq tn^{3/5}$, then $f(xy)\leq W$ for every edge $xy$ of the cover graph of $P_s$. Hence, by Lemma \ref{lemma:3level+}, $P_s$ has at least $W^{-2|L_{m+s}|}e^{-|L_{m+s-1}|+|L_{m+s}|}$ symmetric chain decompositions. If $s>tn^{3/5}$, we only use the fact that $P_s$ has at least 1 symmetric chain decomposition (which is guaranteed by the existence of an SNMF). Thus, we conclude that the number of symmetric chain decompositions of $[t]^n$ is at least
$$\prod_{s=1}^{tn^{3/5}}W^{-2|L_{m+s}|}e^{-(|L_{m+s-1}|+|L_{m+s}|)}<W^{-(1-o_n(1))t^n} e^{-t^n}=n^{(1-o_n(1))t^n}.$$
Here, we used the facts that $|L_m|=O(t^{n-1}/\sqrt{n})$ (see e.g. \cite{book_Anderson}) and $\sum_{s\geq tn^{3/5}}|L_{m+s}|=o_n(t^n)$ (see e.g. \cite{FRT}).
\end{proof}

\noindent
\textbf{Acknowledgments.} We would like to thank Adam Zsolt Wagner for valuable discussions.

\end{document}